\documentclass[11pt,twoside]{amsart}
\usepackage{amsmath}
\usepackage{amsthm}
\usepackage{amsfonts}
\usepackage{amssymb}
\usepackage[all]{xy}
\usepackage{alltt}
\usepackage{enumitem}
\usepackage{stmaryrd}
\usepackage{comment}
\usepackage{xspace}
\usepackage{comment}
\usepackage{mathtools}

\theoremstyle{plain}
\newtheorem{prop}{Proposition}[section]

\newtheorem{thm}[prop]{Theorem}
\theoremstyle{definition}

\newtheorem{remark}[prop]{Remark}

\newcommand{\trank}{\ensuremath{\mathrm{d}}}

\newcommand{\Z}{\ensuremath{\mathbb{Z}}}
\newcommand{\F}{\ensuremath{\mathbf{F}}}

\DeclareMathOperator{\im}{im}

\DeclareMathOperator{\GL}{\mathbf{GL}}
\DeclareMathOperator{\ch}{char}
 
\newcommand{\n}{\noindent}
\newcommand{\map}[3]{\ensuremath{#1 : #2 \longrightarrow #3}}
\newcommand{\dih}{\ensuremath{\mathbf{D}}}
\newcommand{\cyc}{\ensuremath{\mathbf{C}}}
\newcommand{\upt}{\ensuremath{\mathbf{U}}}
\renewcommand{\F}{\ensuremath{\mathbf{F}}}
\renewcommand{\Z}{\ensuremath{\mathbf{Z}}}
\newcommand{\gk}[1]{\ensuremath{\Z_{#1}\rtimes \Z[\cyc_2]}}

\usepackage[margin=1.25in]{geometry}

\usepackage[
        colorlinks=true,urlcolor=blue,linkcolor=blue,citecolor=blue
]{hyperref}

\begin{document}  
 
\title{Fuchs' problem for dihedral groups}
\date{\today}
 
\author{Sunil K. Chebolu}
\address{Department of Mathematics \\
Illinois State University \\
Normal, IL 61790, USA}
\email{schebol@ilstu.edu}

\author{Keir Lockridge} 
\address {Department of Mathematics \\
Gettysburg College \\
Gettysburg, PA 17325, USA}
\email{klockrid@gettysburg.edu}


 
\begin{abstract}
More than 50 years ago, L\'{a}szl\'{o} Fuchs asked which abelian
groups can be the group of units of a ring.
Though progress has been made, the question remains open. One could equally well pose the question for various classes of nonabelian groups.
In this paper, we prove that $\dih_2, \dih_4, \dih_6, \dih_8,$ and $\dih_{12}$ are the only dihedral groups that appear as the group of units of a ring of positive characteristic (or, equivalently, of a finite ring), and $\dih_2$ and $\dih_{4k}$, where $k$ is odd, are the only dihedral groups that appear as the group of units of a ring of characteristic 0.
\end{abstract}
 
\maketitle
\thispagestyle{empty}

\tableofcontents


\section{Introduction}\label{sec:introduction}

L\'{a}szl\'{o} Fuchs poses the following problem in \cite{fuchsprob}: determine which abelian groups are the group of units of a commutative ring. In \cite{PS}, following \cite{gilmer}, all finite cyclic groups which occur as the group of units of a ring are determined. In \cite{fuchs}, we provide an answer to this question for indecomposable abelian groups. Rather than narrowing the class of abelian groups under consideration, one could also broaden the scope of the problem by considering nonabelian groups and noncommutative rings. Call a group $G$ {\bf realizable} if it is the group of units of some ring. In \cite{do1} and \cite{do2} the authors determine which alternating, symmetric, and finite simple groups are realizable.
In the present work, we determine which dihedral groups are the group of units of a ring, and our classification is stratified by characteristic. All rings in this paper are rings with unity. Our results are summarized in the following theorem.

\begin{thm} Let $R$ be a ring. If the group of units $R^\times$ is a dihedral group, then $\mathrm{char}\, R = 0, 2, 3, 4, 6, 8$, or $12.$ The following is a complete list of the dihedral groups which are realizable as the group of units for a ring $R$ with $c = \ch R$. In the third column, we give an example of such a ring in each case.

\begin{center}
\begin{tabular}{| l | l | l |}
\hline

$c = 0$ & $\dih_2$ & $\Z$ \\
& $\dih_{4k}$, $k$ odd & $\gk{k}$ \\ \hline

$c = 2$ & $\dih_2$ &  $\F_2[\dih_2]$ \\
& $\dih_4$ &  $\F_2[\dih_2] \times \F_2[\dih_2]$ \\
& $\dih_6$ & $\mathbf{M}_{2}(\F_2)$ \\
& $\dih_8$ & $\upt_3(\F_2)$ \\ 
& $\dih_{12}$ & $\F_2[\dih_6]$ \\ \hline

$c = 3$ & $\dih_2$ & $\F_3$ \\
& $\dih_4$ & $\F_3 \times \F_3$ \\
& $\dih_{12}$ & $\upt_2(\F_3)$ \\ \hline

$c = 4$ & $\dih_2$ & $\Z_4$ \\
& $\dih_4$ & $\Z_4 \times \Z_4$ \\
& $\dih_8$ & $\mathrm{End}_{\Z}(\cyc_4 \times \cyc_2)$ \\ 
& $\dih_{12}$ & $\Z_4 \times \mathbf{M}_{2}(\F_2)$ \\ \hline

$c = 6$ & $\dih_2$ & $\F_2 \times \F_3$ \\
& $\dih_4$ &  $\F_2 \times \F_3 \times \F_3$ \\
& $\dih_{12} $ & $\F_2 \times \upt_{2}(\F_3)$ \\ \hline

$c = 8$ & $\dih_4$ & $\Z_8$ \\ \hline

$c = 12$ & $\dih_4$ & $\Z_{12}$ \\ \hline
\end{tabular}
\end{center}
\label{main}
\end{thm}

\n The ring $\gk{k}$ is defined in \S \ref{sec:char0}. The group $\cyc_n$ denotes the (multiplicative) cyclic group of order $n$, $\F_{p^k}$ denotes the finite field of order $p^k$ (for $p$ a prime), $\Z_n$ denotes the ring of integers modulo $n$, and $k[G]$ denotes the group algebra of $G$ with coefficients in $k$. The ring $\mathbf{M}_n(k)$ denotes the ring of $n \times n$ matrices with entries in $k$, and $\upt_n(k)$ denotes the ring of $n \times n$ upper triangular matrices with entries in $k$. The group of units of $\upt_3(\F_2)$ is the Heisenberg group of order 8, which is isomorphic to $\dih_8$. The group algebras $\F_2[\dih_2]$ and $\F_2[\dih_6]$ are isomorphic to $\F_2[x]/(x^2)$  and $\F_2[x]/(x^2) \times \mathbf{M}_{2}(\F_2)$, respectively.

\begin{proof} We prove in Proposition \ref{gamma-4k} that the unit group of $\gk{k}$ is $\dih_{4k}$ when $k$ is odd. It is straightforward to check that the remaining rings have the indicated dihedral unit groups; we leave this to the reader to verify. In Proposition \ref{chars}, we show that the characteristic of a ring whose group of units is dihedral must be 0, 2, 3, 4, 6, 8, or 12.  It is clear from the proof that $\dih_4$ is the only dihedral group realizable in characteristic 8 since it is the only dihedral group whose center contains $\Z_8^\times \cong \cyc_2 \times \cyc_2$, and $\dih_4$ is the only dihedral group realizable in characteristic 12 since it is the only dihedral group whose center contains $\Z_{12}^\times \cong \cyc_2 \times \cyc_2$. We eliminate the possibility that any unlisted dihedral groups are the group of units of a ring of characteristic $c$ in Propositions \ref{char0-8k} ($c = 0$); \ref{char2main} and \ref{char2d24} ($c = 2$); \ref{char3} ($c = 3$); \ref{char4main} and \ref{char4d16} ($c = 4$); \ref{char6} ($c = 6$).
\end{proof}

There is a finite ring whose unit group is dihedral if and only if there is a ring of positive characteristic whose unit group is dihedral (see Proposition \ref{reducefinite}). Note also that every dihedral group that occurs as the group of units of a finite ring occurs as the group of units of a ring of characteristic 2, and this is precisely the list of dihedral groups which occur as subgroups of $\GL_2(\Z)$ (see \cite{fg2x2}). The nonabelian dihedral groups appearing as the group of units of a finite ring, $\dih_6, \dih_8,$ and $\dih_{12}$, arise in at least two other interesting situations: they are the only nonabelian dihedral groups isomorphic to their own automorphism groups, and they correspond precisely with the only regular polygons that tile the plane (equilateral triangles, squares, and regular hexagons).

We wish to thank the developers of the open source mathematics software system SageMath; it enabled us to do several computations that helped illuminate various aspects of the proof of Theorem \ref{main}. We are also grateful to the referee for providing detailed and helpful comments.

\section{Preliminaries} \label{sec:prelim}

We will use the following familiar presentation of the dihedral group of order $2n$, generated by a rotation $r$ of order $n$ and a fixed reflection $s$: $$\dih_{2n} = \langle r, s\, | \, r^n = 1, s^2 = 1, srs = r^{-1}\rangle.$$ The following properties of the dihedral group are well known and will be used freely throughout this paper. We will refer to a group $G$ as indecomposable if it is not isomorphic to a direct product of two nontrivial groups, and we will refer to a ring $R$ as indecomposable if it is not isomorphic to a direct product of two nontrivial rings (the trivial ring is the ring $\{0\}$). 

\begin{prop} The dihedral group $\dih_{2n}$ satisfies the following properties. \label{dih-prop}
\begin{enumerate}
\item We have $\dih_2 \cong \cyc_2$ and $\dih_4 \cong \dih_2 \times \dih_2 \cong \cyc_2 \times \cyc_2$.
\item For $n > 2$, the center of $\dih_{2n}$ is $\langle 1 \rangle$ if $n$ is odd and $\langle r^{n/2} \rangle$ if $n$ is even.
\item For $n > 2$, the dihedral group is indecomposable unless $n = 2k$ where $k$ is odd, in which case $$\dih_{2n} \cong \cyc_2 \times \dih_{2k},$$ where both factors are indecomposable. This direct product decomposition is unique up to the order of the factors. \label{dih-prop-decomp}
\item For $n > 2$ and $n$ odd, the only proper normal subgroups of $\dih_{2n}$ are the subgroups of $\langle r \rangle$. For $n > 2$ and $n$ even, there are two additional proper normal subgroups,  $\langle r^2, s\rangle$ and $\langle r^2, rs\rangle$, both of order $n$ and isomorphic to $\dih_n$.
\end{enumerate}
\label{dihedralprop}
\end{prop}

We will say that a group $G$ is realizable in characteristic $m$ if $G$ is the group of units of some ring of characteristic $m$. The next proposition shows that a realizable group is the group of units of a quotient of a group algebra.

\begin{prop}
If $\dih_{2n}$ is realizable in characteristic $m$, then it is the group of units of the ring $\Z_m[\dih_{2n}]/I$ for some two-sided ideal $I$. In particular, $\dih_{2n}$ is the group of units of a finite ring if and only if it is realizable in positive characteristic.
\label{reducefinite}
\end{prop}
\begin{proof}
Let $R$ be a ring of characteristic $m$ whose group of units is isomorphic to $\dih_{2n}$.  There is a ring homomorphism $$\map{\varphi}{\Z_m[\dih_{2n}]}{R}$$ that restricts to the identity map on $\dih_{2n} \leq (\Z_m[\dih_{2n}])^\times$ (note that if $m = 0$, then $\Z_m = \Z$). The image of $\varphi$ is isomorphic to $R' = \Z_m[\dih_{2n}]/I$, where $I$ is a two-sided ideal equal to the kernel of $\varphi$, and $(R')^\times = \dih_{2n}$.
\end{proof}

In this and subsequent sections, we will replace $R$ with the quotient $\Z_m[\dih_{2n}]/I$ appearing in the above proposition as necessary. Usually, the significance of this replacement is simply that $R$ may be assumed to be generated by its units and finite when $m > 0$. We next determine the possible characteristics of a ring with dihedral units.

\begin{prop} If $R$ is a ring whose group of units is isomorphic to a dihedral group $\dih_{2n}$, then the characteristic of $R$ is $0, 2, 3, 4, 6, 8$ or $12$. If $n > 1$ is odd, then the characteristic of $R$ must be $2$.
\label{chars}
\end{prop}
\begin{proof}
Suppose $R$ is a ring of characteristic $m > 0$ with $R^\times = \dih_{2n}$. Since $\Z_m$ is a central subring of $R$, $\Z_m^\times$ is a subgroup of the center of $\dih_{2n}$. This center is trivial if $n > 2$ is odd, isomorphic to $\cyc_2$ if $n = 1$ or $n > 2$ is even, and isomorphic to $\cyc_2 \times \cyc_2$ if $n = 2$.  If $m > 12$, we have $|\Z_m^\times| > 4$, so $\mathrm{char}\, R \leq 12$. If the center has order at most 2, then $n = 2, 3, 4,$ or $6$. For $\dih_4 = \cyc_2 \times \cyc_2$, the ring may also have characteristic 8 (indeed, $\Z_8^\times \cong \dih_4$) or 12 (since $\Z_{12}^\times \cong \dih_4$). For $n > 1$ note that $-1 \in R$ is a central element of multiplicative order 2, but if $n$ is odd there is no such element of $\dih_{2n}$, and hence for $n$ odd the characteristic must equal 2.
\end{proof}

To motivate the next two propositions, consider a ring $R$ of prime characteristic $p$ whose group of units contains an element of order $p^r$. We would like to know whether this imposes any interesting restrictions on the size of $R^\times$. There is a ring homomorphism
\begin{equation}
\map{\varphi}{\F_p[x]/(x^{p^r} - 1)}{R}
\label{phimap}
\end{equation}
sending $x$ to a unit of order $p^r$. The domain of $\varphi$ is isomorphic to $\F_p[x]/(x^{p^r})$, so $S = \im \varphi$ is a subring of $R$ isomorphic to $\F_p[x]/(x^n)$ for some $n \leq p^r$. In the next proposition, we compute the group of units of $\F_p[x]/(x^n)$.

\begin{prop} The group of units of $\mathbf{F}_p[x]/(x^n)$ is isomorphic to $$U_n = \cyc_{p-1} \oplus \left(\bigoplus_{1 \leq k < 1 + \log_p n} \cyc_{p^k}^{\left\lceil \frac{n}{p^{k-1}} \right\rceil - 2\left\lceil \frac{n}{p^k} \right\rceil + \left\lceil \frac{n}{p^{k+1}} \right\rceil}\right).$$ 
\label{unitprop}
\end{prop}
\begin{proof}
First, observe that an element of $R = \mathbf{F}_p[x]/(x^n)$ is a unit if and only if its constant term is nonzero. Thus, $R^\times \cong \cyc_{p-1} \times B$, where $B$ is the subgroup of units with constant term 1. Further, $|B| = p^{n-1}$ so every element of $B$ has order a power of $p$.

Now, let $\alpha_{p^k}$ denote the number of units in $B$ of order dividing $p^k$.  A non-identity element $f(x) = 1 + c_jx^j + \cdots$, with $c_j \neq 0$, is a unit of order dividing $p^k$ if and only if $p^kj \geq n$. Hence, $$\alpha_{p^k} = p^{n - \left\lceil\frac{n}{p^k}\right\rceil}.$$ Let $\beta_{p^k}$ denote the number of factors in the canonical decomposition of $B$ isomorphic to $\cyc_{p^k}$. Then, $$\alpha_{p^k} = p^{\beta_p} \cdot (p^2)^{\beta_{p^2}} \cdot \cdots \cdot (p^{k-1})^{\beta_{p^{k-1}}} \cdot (p^k)^{\beta_{p^k} + \beta_{p^{k+1}} + \,\cdots}.$$ Solving for $\beta_{p^k}$, one finds that $$\beta_{p^k} = \log_p\left[\frac{\alpha_{p^k}^2}{\alpha_{p^{k-1}} \alpha_{p^{k+1}}}\right].$$ The formula given in the statement now follows. 
\end{proof}

For a finite abelian group $G$, let $\trank(G)$ denote the number of factors in the canonical decomposition of $G$ as a direct product of cyclic groups of prime power order. The next proposition gives a useful lower bound on $\trank(G)$ for some finite abelian subgroup $G \leq R^\times$ and gives criteria for the presence of noncyclic abelian subgroups.

\begin{prop} Let $R$ be a ring of prime characteristic $p$ whose group of units contains an element of order $p^r \geq p^2$. \label{rankprop}
\begin{enumerate}
\item If $p > 2$, then $R^\times$ has a noncyclic finite abelian subgroup $G$ such that $\trank(G) \geq 1 + (p-1)p^{r-2}$. \label{rankpropodd}
\item If $p = 2$, then $R^\times$ has a finite abelian subgroup $G$ such that $\trank(G) \geq 2^{r-2}$. If $r \geq 3$, then $G$ is noncyclic. \label{rankpropeven}
\end{enumerate}
\end{prop}
\begin{proof}
As explained in the paragraph before Proposition \ref{unitprop} where the ring homomorphism $\varphi$ is defined, the domain of $\varphi$ is isomorphic to $\F_p[x]/(x^{p^r})$, so $S = \im \varphi$ is a subring of $R$ isomorphic to $\F_p[x]/(x^n)$ for some $n \leq p^r$. The ring $S$ must contain an element of order $p^r$, so in light of the decomposition of $U_n$ appearing in Proposition \ref{unitprop} we must have $1 \leq r < 1 + \log_p n$. Thus $n \geq p^{r -1} + 1$. The quotient map $S \longrightarrow \F_p[x]/(x^{p^{r-1} + 1})$ is surjective on units since its kernel, the ideal $(x^{p^{r-1} + 1})$, is a nilpotent ideal of $S$. Consequently, $U_{p^{r-1} + 1}$ is a quotient of a subgroup $G$ of $R^\times$. Since taking quotients cannot increase the number of factors in the canonical decomposition of a finite abelian group, $\trank(G) \geq \trank(U_{p^{r-1} + 1})$. Using Proposition \ref{unitprop}, one can check that $$U_{p^{r-1} + 1} = \cyc_{p-1} \times \cyc_{p^{r-1}}^{p-1} \times \cyc_{p^{r-2}}^{(p-1)^2} \times \cyc_{p^{r-3}}^{p(p-1)^2} \times \cdots \times \cyc_{p^2}^{p^{r-4}(p-1)^2} \times \cyc_p^{p^{r-3}(p-1)^2}.$$ If $p > 2$, $\trank(U_{p^{r-1} + 1}) = 1 + (p-1)p^{r-2}$. If $p = 2$, $\trank(U_{p^{r-1} + 1}) = 2^{r-2}$.

When $p > 2$, $\trank(U_{p^{r-1} + 1}) \geq 3$, and if $p = 2$ and $r \geq 3$, then $\trank(U_{p^{r-1} + 1}) \geq 2$. In both cases, $\trank(H) \geq 2$, where $H$ is the $p$-torsion summand of $U_{p^{r-1} + 1}$. It follows that $G$ is a noncyclic abelian subgroup.
\end{proof}

Before turning to the proof of Theorem \ref{main} for rings of positive characteristic in the next section, we collect a few results from the theory of noncommutative rings. Let $R$ be a ring with a finite group of units and let $I$ be a two sided ideal contained in the Jacobson radical $J$ of $R$. The Jacobson radical satisfies $J = \{ x \in R\, | \, 1 + RxR \subseteq R^\times\}$. Now, if $a + I$ is a unit of $R/I$, then there is an element $b \in R$ such that $ab - 1 \in I \subseteq J$, so $ab$ is a unit of $R$ and therefore $a$ has a right inverse. Similarly, $a$ also has a left inverse and hence $a$ must be a unit of $R$. This proves that the quotient map $R \longrightarrow R/I$ is surjective on units.  The remainder of the following proposition follows easily from this fact.

\begin{prop} Let $R$ be a ring with a finite group of units and let $I$ be a two-sided ideal of $R$ contained in the Jacobson radical. Then:
\begin{enumerate}
\item The quotient map $\map{\varphi}{R}{R/I}$ induces a surjective group homomorphism $$\map{\varphi^\times}{R^\times}{(R/I)^\times}.$$
\item The set $1 + I = \ker \varphi^\times$ is a normal subgroup of $R^\times$, and $I$ is finite since $R^\times$ is finite.
\item We have $|R^\times| = |(R/I)^\times||\ker \varphi^\times| =  |(R/I)^\times||I|$ and hence $|I|$ divides both $|R|$ and $|R^\times|$.
\end{enumerate}
\label{junits}
\end{prop}

The next proposition is part of the theory of idempotents in noncommutative rings; see \cite[\S 21, \S 22]{lam} for more details.  Central idempotents in such rings may be used to find block decompositions, and it is natural to ask whether information from a block decomposition of a quotient may give information about the original ring. If $R$ is a ring and $I$ is a nilpotent two-sided ideal contained in the Jacobson radical of $R$, then there is a bijection between the set of central idempotents of $R$ and the central idempotents of $R/I^2$ (see \cite[22.9]{lam}). We only need the following proposition, which is a corollary to this fact.

\begin{prop} If $R$ is a ring and $I$ is a nilpotent two-sided ideal contained in the Jacobson radical of $R$, then $R$ is indecomposable if and only if $R/I^2$ is indecomposable.
\label{j2indecomp}
\end{prop}

Finally, we need a structure theorem for finite rings with trivial Jacobson radical. A ring with trivial Jacobson radical is called semiprimitive, and any Artinian semiprimitive ring is semisimple.  The Artin-Wedderburn Theorem then implies the ring is a product of matrix rings over division rings. If the ring is finite, then each division ring is finite and therefore a field by Wedderburn's Little Theorem. This proves the following proposition.

\begin{prop} Let $R$ be a finite ring of prime characteristic $p$. If $R$ has trivial Jacobson radical, then it is a finite product of matrix rings of the form $\mathbf{M}_m(\F_{p^k})$ where $m$ and $k$ are positive integers.
\label{jzero}
\end{prop}

\section{Rings of positive characteristic} \label{sec:finite}

The goal of this section is to prove Theorem \ref{main} for rings of positive characteristic. To begin, we eliminate the unlisted dihedral groups in the characteristic 2 case.

\begin{prop} Let $n$ be a positive integer such that $8 \nmid n$.  If $R$ is a finite ring of characteristic $2$ with trivial Jacobson radical and unit group $\dih_{2n}$, then $n = 3$.
\label{char2j0}
\end{prop}
\begin{proof}
Suppose $R$ is an finite ring of characteristic 2 with trivial Jacobson radical such that $R^\times = \dih_{2n}$ for some $n$, where $8 \nmid n$.  By Proposition \ref{jzero}, $R$ is a product of matrix rings of the form $\mathbf{M}_m(\F_{2^k})$.  The dihedral group $\dih_{2n}$ is indecomposable unless $n = 2v$ for $v$ odd, in which case $$\dih_{4v} = \cyc_2 \times \dih_{2v}$$ by Proposition \ref{dih-prop} (\ref{dih-prop-decomp}). Since this direct product decomposition is unique (up to the order of the factors) and $\cyc_2$ cannot be the group of units of any such matrix ring, we must have $\dih_{2n} \cong \GL_m(\F_{2^k})$. It is well known that the order of $\GL_m(\F_{p^k})$ is $$(p^{mk} - 1)(p^{mk} - p^k) \cdot \cdots \cdot (p^{mk} - p^{(m-1)k}).$$ The exponent of $p$ in the prime factorization of this quantity is $km(m-1)/2$. Since $p = 2$ and $8 \nmid n$, we have $km(m-1) \leq 6$.  When $m = 1$, $\GL_m(\F_{2^k}) = \F_{2^k}^\times$ has odd order, but the dihedral group has even order, so $m > 1$. If $k = 3$, then $m = 2$; the center of $\GL_2(\F_{8})$ is $\F_8^\times = \cyc_7$, which is not the center of any dihedral group. If $k = 2$, then $m = 2$; the center of $\GL_2(\F_4)$ is $\F_4^\times = \cyc_3$, which is not the center of any dihedral group. If $k = 1$, then $m = 2$ or 3; if $m = 2$, then we obtain $\GL_2(\F_2) \cong \dih_6$. If $m = 3$, then $\GL_3(\F_2)$ has trivial center and order 168, but the center of $\dih_{168}$ has order 2.
\end{proof}

\begin{prop} If $\dih_{2n}$ is realizable in characteristic $2$, then $n$ is a divisor of $12$. \label{char2main}
\end{prop}
\begin{proof}
Suppose $R$ is a ring of characteristic 2 whose group of units is $\dih_{2n}$. If $8 \, | \, n$, then $R$ contains a unit of order $2^3$. By Proposition \ref{rankprop} (\ref{rankpropeven}), $\dih_{2n}$ contains a noncyclic abelian subgroup, but no such subgroup exists for $n \neq 2$. Hence, $8 \nmid n$.

As observed in Proposition \ref{reducefinite}, we may assume $R$ is finite. Note further that $R$ and its Jacobson radical $J$ have orders which are powers of 2.  But $|J|$ must also divide $|R^\times|$ so $|J| = 1, 2, 4$ or 8.

If $|J| = 1$, then $\dih_{2n} = \dih_6$ by Proposition \ref{char2j0}.

If $|J| = 2$, then $1 + J$ is a normal subgroup of order 2 in $\dih_{2n}$.  If this subgroup lies in the rotation subgroup then 2 divides $n$ and $\dih_{n} = (R/J)^\times$ is realizable in characteristic 2 by a ring with trivial Jacobson radical, forcing $n = 6$ by Proposition \ref{char2j0}. Otherwise, $1+J$ is one of the two order $n$ subgroups  $\langle r^2, s\rangle$ or $\langle r^2, rs\rangle$, forcing $n = 2$. Finally, we may have $1 + J = \dih_{2n}$ forcing $n = 1$.

If $|J| = 4$, then $1 + J$ is a normal subgroup of order 4 in $\dih_{2n}$. If this subgroup lies in the rotation subgroup then 4 divides $n$ and $\dih_{n/2} = (R/J)^\times$ is realizable in characteristic 2 by a ring with trivial Jacobson radical, forcing $n = 12$ by Proposition \ref{char2j0}.  If $1 + J$ does not consist entirely of rotations, then $n$ must be even and $1+J = \langle r^2, s\rangle$ or $\langle r^2, rs\rangle$. These subgroups have order $n$, forcing $n = 4$. If $1 + J = \dih_{2n}$, then $n = 2$.

If $|J| = 8$, then since $8 \nmid n$ the only possibility is that $1 + J = \dih_{2n}$ and $n = 4$.
\end{proof}

In the characteristic 2 case, it remains to eliminate $\dih_{24}$. To that end, we first prove the following proposition.

\begin{prop} If $R$ is a ring of characteristic $2$ generated by its units with $R^\times \cong \dih_{12}$, then $R$ is decomposable. \label{d12decomp}
\end{prop}
\begin{proof}
Let $R$ be a a ring of characteristic 2 generated by its units with $R^\times \cong \dih_{12}$. The ring $R$ is a quotient of $\F_2[\dih_{12}] $.  The element $r^2 + r^4$ is a central idempotent of $\F_2[\dih_{12}] $; we therefore have a decomposition $$\F_2[\dih_{12}] = (r^2 + r^4) \times (1 + r^2 + r^4) = R_1 \times R_2.$$ In the ring $R_2$, $r^2 + r^4 = 0$, so $r^2 = 1$ in $R_2$. It is straightforward to check that $R_2$ is a quotient of $\F_2[\dih_4] \cong \F_2[x, y]/(x^2, y^2)$, a ring with 16 elements and 8 units of order dividing 2.

Assume to the contrary that $R$ is an indecomposable ring. Then $R$ is either a quotient of $R_1$ or $R_2$. If $R$ is a quotient of $R_2$, then it is a quotient of $\F_2[\dih_4]$.  Since $\F_2[\dih_4]$ has no units of order 3, $R$ must be a quotient of $R_2$ of order at most 8.  But no such ring can have $\dih_{12}$ as its group of units. Hence, $R$ is a quotient of $R_1$. This means $1 + r^2 + r^4 = 0$ in $R$ or, equivalently, $r^2 + r^4 = 1$. Now consider the element $a = s + sr$. Direct calculation shows that $a^4 = 1$ in $R$. But since $\dih_{12}$ has no elements of order 4, we must have $a^2 = 1$. Now, $1 = a^2 = r + r^5$ and hence $r^3 = r^4 + r^2 = 1$ in $R$, a contradiction.
\end{proof}

\begin{prop} The dihedral group $\dih_{24}$ is not realizable in characteristic $2$.
\label{char2d24}
\end{prop}
\begin{proof}
If $\dih_{24}$ is realizable in characteristic 2, then by Proposition \ref{reducefinite} there exists a finite ring $R$ of characteristic 2, generated by its units, such that $R^\times = \dih_{24}$. We may further assume $R$ is indecomposable: if $R \cong R_1 \times R_2$, then since $\dih_{24}$ indecomposable, without loss of generality $R_1^\times = \dih_{24}$ and we may replace $R$ with $R_1$ (which is also generated by its units). Since $R$ is finite, we may repeat this step until we obtain an indecomposable ring. 

Let $J$ denote the Jacobson radical of $R$. Since $24 = 2 \cdot 4\cdot 3$, as in the proof of Proposition \ref{char2main} we must have $|J| = 1, 2,$ or 4 (since $\dih_{24}$ has no normal subgroup of order 8, $|J| \neq 8$). We cannot have $|J| = 1$ by Proposition \ref{char2j0}. If $|J| = 2$, then $1 + J = \langle r^6 \rangle$ and $R/J$ is has trivial Jacobson radical and unit group $\dih_{12}$, which is impossible by Proposition \ref{char2j0}.

It remains to consider the possibility that $|J| = 4$. In this case, $1 + J = \langle r^3 \rangle$. Since $J$ is nilpotent, by Proposition \ref{j2indecomp} we obtain that since $R$ is indecomposable, so is $R/J^2$.  Since $1 + J^2 = \langle r^6 \rangle$ and the map $R \longrightarrow R/J^2$ is surjective on units, $R/J^2$ is an indecomposable ring generated by its units with unit group $\dih_{12}$. This is impossible by Proposition \ref{d12decomp}.
\end{proof}

We now turn our attention to rings of characteristic 3.

\begin{prop} If $n \neq 1, 2$ or $6$, then $\dih_{2n}$ is not realizable in characteristic $3$. \label{char3}
\end{prop}
\begin{proof}
Suppose $R$ is a ring of characteristic 3 with $R^\times = \dih_{2n}$. If $9 |\, n$, then $R$ contains a unit of order $3^2$. By Proposition \ref{rankprop} (\ref{rankpropodd}), $\dih_{2n}$ must contain a noncyclic abelian subgroup, but no such subgroup exists for $n \neq 2$. Hence, $9 \nmid n$.

Since we may assume $R$ is finite, the Jacobson radical of $R$ has order 1 or 3 since $9 \nmid n$. If $|J| = 1$, then $R$ is a product of matrix rings of the form $\mathbf{M}_m(\F_{3^k})$. Suppose $\dih_{2n} = \GL_m(\F_{3^k})$. Using the formula for $|\GL_m(\F_{3^k})|$ discussed in the proof of Proposition \ref{char2j0}, we see that the exponent of 3 in $|\GL_m(\F_{3^k})|$ is $km(m-1)/2$. Since $9 \nmid n$, we have $km(m-1) \leq 2$.  If $m = 1$, then the dihedral group $\dih_{2n} = \F_{3^k}^\times$ is cyclic, forcing $k = 1$ and $n = 1$.  If $m = 2$, then $k = 1$; however, $\GL_2(\F_3)$ has more than one subgroup of order 3 (take, for example, upper triangular matrices with $1$'s on the diagonal and lower triangular matrices with $1$'s on the diagonal), whereas $\dih_{48}$ has a unique subgroup of order 3 generated by $r^8$. To summarize, if $\dih_{2n}$ is indecomposable and realizable in characteristic 3 by a ring with trivial Jacobson radical, then $\dih_{2n} = \dih_2$.  If $\dih_{2n}$ is decomposable, then $n = 2n'$, where $n'$ is odd, $\dih_{2n} \cong \dih_2 \times \dih_{2n'}$, and $\dih_{2n'}$ is indecomposable. So if $\dih_{2n}$ is decomposable and realizable in characteristic 3 by a ring with trivial Jacobson radical, then $\dih_{2n} = \dih_2 \times \dih_2 = \dih_4$.

Now suppose $|J| = 3$. Then $(R/J)^\times \cong \dih_{2n/3}$ where $R/J$ has trivial Jacobson radical. Hence, if $\dih_{2n}$ is realizable, then $n = 3$ or $n = 6$. To finish the proof, we must eliminate $\dih_6$. If $R^\times \cong \dih_6$, then $R$ has a subring isomorphic to a quotient of $\F_3[x]/(x^3 - 1) \cong \F_3[x]/(x^3)$, where $x$ corresponds to $r$. By Proposition \ref{unitprop}, the unit group of this subring is $\cyc_2 \times B$, where $B$ is a finite abelian 3-group, necessarily nontrivial since the subring must contain an element of order 3. The group $R^\times$ must therefore contain an element of order 6, but $\dih_6$ has no such element. This completes the proof.
\end{proof}

Together with the classification in characteristics 2 and 3, the following proposition completes the classification in characteristic 6. Condition (\ref{char6-1}) gives rise to all the dihedral groups in characteristic 6 in Theorem \ref{main}; conditions (\ref{char6-2}) and (\ref{char6-3}) generate no other examples (though the rings therein are not generally isomorphic to the rings appearing in condition (\ref{char6-1})).

\begin{prop}  If $\dih_{2n}$ is the group of units of a ring $R$ of characteristic $6$ then either
\begin{enumerate}
\item $\dih_{2n} = (\F_2 \times S)^\times$, where $\ch S = 3$;\label{char6-1}
\item $2n = 4v$, where $v$ is odd, and $\dih_{2n} = (S \times \F_3)^\times$, where $\ch S = 2$; or\label{char6-2}
\item $2n = 4v$, where $v$ is odd, and $\dih_{2n} = (\F_2[x]/(x^2) \times S)^\times$ where $\ch S = 3$.\label{char6-3}
\end{enumerate}
\label{char6}
\end{prop}
\begin{proof}
Suppose $R$ is a ring of characteristic 6 with unit group $\dih_{2n}$.  Then $R \cong R_1 \times R_2$, where $R_1$ has characteristic 2, $R_2$ has characteristic 3, and $\dih_{2n} = R_1^\times \times R_2^\times$.  The group $R_2^\times$ cannot be trivial, so one possibility is that $R_2^\times = \dih_{2n}$ and condition (\ref{char6-1}) holds with $S = R_2$. Otherwise, $\dih_{2n}$ must be decomposable. In this case, $n = 2v$, where $v$ is odd, and $\dih_{2v}$ is indecomposable with $\dih_{2v} = R_1^\times$ or $\dih_{2v} = R_2^\times$. In the former case, $\dih_{2n} = (R_1 \times \F_3)^\times$ and condition (\ref{char6-2}) holds, and in the latter case $\dih_{2n} = (\F_2[x]/(x^2) \times R_2)^\times$ and condition (\ref{char6-3}) holds.
\end{proof}

To complete the proof of Theorem \ref{main} for rings of positive characteristic, only the characteristic 4 case remains.

\begin{prop} If $\dih_{2n}$ is realizable in characteristic $4$, then $n = 1, 2, 4, 6$ or $8$.\label{char4main}
\end{prop}
\begin{proof}
Suppose $R$ is a ring of characteristic 4 with $R^\times = \dih_{2n}$.  For any $t \in R$, we have $(1 + 2t)^2 = 1$.  This implies that $1 + 2t$ is an element of order dividing 2 in $\dih_{2n}$.  Thus, $$1 + (2) \subseteq \{g \in \dih_{2n}\, | \, |g| \leq 2\}.$$ Consider the quotient map $\map{\varphi}{R}{R/(2)}$. Note that $\map{\varphi^\times}{R^\times}{(R/(2))^\times}$ is surjective since $(2)$ is a nilpotent ideal and $$1 + (2) = \ker \varphi^\times.$$ This kernel cannot be trivial for otherwise $2 = 0$ in $R$ and $R$ has characteristic 2, not 4. Hence, $\ker \varphi^\times$ must be a nontrivial normal subgroup of $\dih_{2n}$ consisting entirely of elements of order dividing 2. If $n \neq 1, 2, 4$, then we must have $n$ even and $\ker \varphi^\times = \langle r^{n/2}\rangle$. Consequently, $(R/(2))^\times \cong \dih_n$.  This forces $n$ to be 8 or an even divisor of 12 by Propositions \ref{char2main} and \ref{char2d24}.

In the case $n =12$, we may assume (as in the proof of Proposition \ref{char2d24}) that $R$ in the above argument is finite, indecomposable, and generated by its units. Note that the Jacobson radical $J$ of $R$ contains the nilpotent ideal $(2)$. Now the ring $R/(2)$ is a ring of characteristic 2 whose unit group is isomorphic to $\dih_{12}$. By the proof of Proposition \ref{char2main}, the Jacobson radical of $R/(2)$ must have size 2. Hence, $|J| = 4$. Since $1 + J$ is a normal subgroup of $\dih_{24}$ of order 4, we must have $1 + J = \langle r^3 \rangle$ and $1 + J^2 = \langle r^6 \rangle$, so $J^2 = (2)$. Since $R$ is indecomposable, $R/J^2 = R/(2)$ is indecomposable by Proposition \ref{j2indecomp}. This contradicts Proposition \ref{d12decomp}.
\end{proof}

Finally, we eliminate $\dih_{16}$ in characteristic 4.

\begin{prop} 
The dihedral group $\dih_{16}$ is not realizable in characteristic $4$. \label{char4d16}
\end{prop}
\begin{proof}
Assume to the contrary that there is a ring $R$ of characteristic 4 with $R^\times = \dih_{16}$. As observed earlier, $-1$ is a central unit of order 2 so $r^4 = -1$ in $R$. We therefore have a ring homomorphism $$\map{\varphi}{\Z_4[x]/(x^4 + 1)}{R}$$ sending $x$ to $r$. Let $S$ denote the image of $\varphi$. We claim that $2 = 0$ in $S$ forcing the characteristic of $R$ to be $2$, a contradiction.

Since $S$ is a commutative subring of $R$ containing $\langle r \rangle = \cyc_8$, we must have $S^\times = \cyc_8$ since this is the only abelian subgroup of $\dih_{16}$ containing $\cyc_8$. Hence, $1$ and $-1$ are the only units of order dividing 2 in $S$. Since $(1 + 2r)^2 = 1$, we must either have $1 + 2r = 1$, in which case $2 = 0$ since $r$ is a unit, or $1 + 2r = -1$, in which case $2r - 2 = 0$.

Now, using the fact that $r^4 +1 = 0$ and $2r - 2 = 0$, we have $(1 + (1+r)^3)^2 = 1$. Thus, either $1 + (1 + r)^3 = 1$, in which case $0 = (1+ r)^4 + 2 = 2$, or $1 + (1+r)^3 = -1$, in which case $0 = (1 + r)^4 + 2 = 2(1+r) + 2 = 2$.
\end{proof}

\section{Rings of characteristic zero}\label{sec:char0}

Suppose $n > 1$ and $\dih_{2n}$ is realizable in characteristic 0. By Proposition \ref{chars}, $n$ must be even. We will begin by proving that $\dih_{4k}$ is realizable when $k \geq 1$ is odd. Let $s$ denote the generator of $\cyc_2 = \langle s \rangle$. The group ring $\Z[\cyc_2] \cong \Z[s]/(s^2 -1)$ has unit group $\{1, -1, s, -s\} \cong \dih_2 \times \dih_2 \cong \dih_4$. We now define a ring $\gk{k}$ that contains $\Z[\cyc_2]$ as a subring and has unit group isomorphic to $\cyc_2 \times \dih_{2k}$. Define maps $$\map{S, D}{\Z[\cyc_2]}{\Z_k}$$ by $S(a + bs) = a + b \text{ mod } k$ and $D(a + bs) = a - b \text{ mod } k$. These maps are ring homomorphisms since they are evaluation at $s = 1$ and $s = -1$, respectively, followed by reduction modulo $k$. As an additive abelian group, let $$\gk{k} = \Z_k \oplus \Z[\cyc_2],$$ and define a product on $\gk{k}$ by $$(t, u)(t', v) = (tD(v) + t'S(u), uv).$$ It is straightforward to check that this binary operation has identity $1 = (0, 1)$ and is both associative and distributive.

\begin{prop} For any positive integer $k$, the ring $\gk{k}$ has unit group isomorphic to $\cyc_2 \times \dih_{2k}$. In particular, if $k$ is odd then the dihedral group $\dih_{4k}$ is realizable in characteristic $0$. \label{gamma-4k}
\end{prop}
\begin{proof}
We will show that $(\gk{k})^\times$ is isomorphic to $\langle -1 \rangle \times H = \pm H$, where $H \leq (\gk{k})^\times$ is isomorphic to $\dih_{2k} = \langle r, s \rangle$. In $\gk{k}$, write $r = (1, 1)$, $s = (0, s)$, and $\pm 1 = (0, \pm 1)$.

The element $r$ satisfies $r^i = (i, 1)$ for any integer $i$, so $r$ is a unit of order $k$. Now, $$sr = (0, s)(1, 1) = (1, s) = (-1, 1)(0, s) = r^{-1}s.$$ The subgroup of $(\gk{k})^\times$ generated by $r$ and $s$ has $2k$ elements and satisfies the relations $s^2 = 1, r^k = 1$, and $sr = r^{-1}s$; it is therefore isomorphic to $\dih_{2k}$. This proves that $(\gk{k})^\times$ contains $\pm \dih_{2k}$ as a subgroup. If $(t, u)$ is a unit of $\gk{k}$, then $u$ must be a unit in $\Z[s]/(s^2 -1)$, so $(\gk{k})^\times$ has at most $4k$ elements. Hence $(\gk{k})^\times = \pm \dih_{2k}$, as desired.
\end{proof}

\begin{remark}
The construction of $\gk{k}$ is an example of a more general phenomenon. Let $A$ and $B$ be rings and let $\map{f, g}{A}{Z(B)}$ be ring homomorphisms, where $Z(B)$ denotes the center of $B$. We may define a ring $T = B \rtimes A$ as follows. As an additive abelian group, let $T = B \oplus A$. Define a multiplication operation on $T$ by  $$(b, a)(b', a') = (bf(a') + b'g(a), aa').$$ The identity for this operation is $1 = (0, 1)$. If $(b, a)$ is a unit of $T$, then $a$ is a unit of $A$ and $$(b, a)^{-1} = (-bg(a^{-1})f(a^{-1}), a^{-1}).$$ The group $A^\times$ is naturally a subgroup of $T^\times$ as $(0, A^\times)$, and the group $B' = (B, +)$ is naturally a normal subgroup of $T^\times$ as $(B, 1)$. Further, $T^\times =  B' \cdot A^\times$ and $B' \cap A^\times = \{ (0, 1)\}$. Hence, $T^\times \cong B' \rtimes A^\times$. The homomorphism $\map{\varphi}{A^\times}{\mathrm{Aut}(B')}$ that determines the semidirect product is given by $\varphi_a(b) = g(a)f(a^{-1})b$. In the case of $\gk{k}$ above, $\varphi_{-1}(t) = t$ and $\varphi_s(t) = -t$, so $(\gk{k})^\times \cong \Z_k \rtimes \langle s, -1 \rangle \cong \dih_{2k} \times \langle -1 \rangle \cong \cyc_2 \times \dih_{2k}$.
\end{remark}

It remains to prove that $\dih_{2n}$ is not realizable in characteristic zero when $n = 4k$. Suppose $R$ is a ring of characteristic zero with unit group $\dih_{8k}$. Since $r^{n/2} = r^{2k}$ is the unique central unit of order 2, we must have $-1 = r^{2k} = (r^k)^2$. Let $i = r^k$. In $R$, the following identities are satisfied:
\[ \begin{aligned}
i^2 &= -1\\
s^2 &= 1\\
si & = sr^k = r^{-k}s = -is.
\end{aligned}\]
Let $Q$ denote the split quaternion ring over the integers. In more detail, $Q$ is the four dimensional free $\Z$-module with basis $\{1, i, s, is\}$ and whose multiplication is determined by the three identities above. This variant of Hamilton's quaternions was first described by Cockle in \cite{cockle}. We now have a ring homomorphism $$\map{\varphi}{Q}{R}$$ sending $i$ and $s$ in $Q$ to $i$ and $s$ in $R$.

Every element of $Q$ has the form $u + vs$ where $u, v \in \Z[i]$. For $\alpha = u + vs \in Q$, let $$\widehat{\alpha} = \bar{u} - vs$$ and define $$\map{N}{Q}{\Z}$$ by $$N(\alpha) = \alpha \widehat{\alpha} = |u|^2 - |v|^2,$$ where $|\cdot|$ denotes the usual norm for the complex numbers. The next proposition summarizes the relevant properties of the norm $N$ and the conjugation operation $\alpha \mapsto \widehat{\alpha}$. For $\alpha = a + bi + cs + dis \in Q$, call the integer $a$ the real part of $\alpha$, denoted $\mathrm{Re}\,\alpha$.

\begin{prop} Let $Q$ be the split quaternion ring over $\Z$ with conjugation operation $\alpha \mapsto \widehat{\alpha}$ and norm $N$. Take $\alpha, \beta \in Q$. \label{N}
\begin{enumerate}
\item We have the identities\[\begin{aligned} \widehat{\alpha \beta} &= \widehat{\beta}\widehat{\alpha}, \\ \widehat{\alpha + \beta} &= \widehat{\alpha} + \widehat{\beta}, \text{and}\\ \widehat{\widehat{\alpha}} &= \alpha. \end{aligned}\] \label{hat-ids}
\item $N$ is multiplicative; i.e., $N(\alpha \beta) = N(\alpha)N(\beta)$.\label{multiplicative}
\item The element $\alpha$ is a unit if and only if $N(\alpha) = \pm 1$.\label{unit}
\item If $N(\alpha) = 1$ and $N(\alpha - 1) = 0$, then $\mathrm{Re}\, \alpha = 1$.\label{oddity}
\end{enumerate}
\end{prop}
\begin{proof}
Statements (\ref{hat-ids}), (\ref{multiplicative}), and (\ref{unit}) are standard (see, for example, \cite{bial}). For (\ref{oddity}), we take $\alpha \in Q$ and compute:
\[\begin{aligned}
0 = N(\alpha - 1) &= (\alpha - 1)(\widehat{\alpha - 1})\\
&= (\alpha - 1)(\widehat{\alpha} - 1)\\
&= N(\alpha) - (\alpha + \widehat{\alpha}) + 1\\
&= 1 - 2\mathrm{Re}\, \alpha + 1\\
&= 2(1-\mathrm{Re}\, \alpha).
\end{aligned}\]
This forces $\mathrm{Re}\, \alpha = 1$, as desired.
\end{proof}

We now prove that $\dih_{8k}$ is not realizable in characteristic zero.

\begin{prop}
If $4$ divides $n$ then the dihedral group $\dih_{2n}$ is not realizable in characteristic $0$. \label{char0-8k}
\end{prop}
\begin{proof}
Assume to the contrary that there is a ring $R$ of characteristic zero with $R^\times \cong \dih_{8k}$. As shown above, there is a ring homomorphism $\map{\varphi}{Q}{R}$, where $Q$ is the split quaternion ring over the integers. Consider the element $$z = (1 + 4i) + (3 + 3i)s \in Q.$$ Since $N(z) = 17 - 18 = -1$, $z$ is a unit in $Q$ and therefore maps to a unit in $R$. We then have $$0 = \varphi(z)^{8k} - 1 = \varphi(z^{8k} - 1)$$ since $|R^\times| = 8k$. Thus, $$0 = \varphi((z^{8k} - 1)(\widehat{z^{8k} - 1})) = \varphi(N(z^{8k} - 1)).$$ If the integer $N(z^{8k} - 1)$ is nonzero, then we will obtain a contradiction since $\mathrm{char}\, R = 0$. By direct computation, $z^2 = 2z + 1$ and so $z^{j} = 2z^{j-1} + z^{j-2}$ for all $j \geq 2$. Consequently, $$\mathrm{Re}\, z^j = 2\mathrm{Re}\, z^{j-1} + \mathrm{Re}\, z^{j-2}.$$ By induction, $\mathrm{Re}\, z^j > 1$ for $j \geq 2$. In particular, we have $\mathrm{Re}\, (z^{8k}) \neq 1$ and $N(z^{8k}) = (-1)^{8k} = 1$ so $N(z^{8k} - 1) \neq 0$ by Proposition \ref{N} (\ref{oddity}).
\end{proof}


\begin{thebibliography}{DO14b}


\bibitem[CL15]{fuchs}
Sunil~K. Chebolu and Keir Lockridge.
\newblock Fuchs' problem for indecomposable abelian groups.
\newblock {\em J. Algebra}, 438:325--336, 2015.

\bibitem[Coc49]{cockle}
James Cockle.
\newblock On Systems of Algebra involving more than one Imaginary.
\newblock {\em Philosophical Magazine} (series 3) 35 (1849), pp. 434-435.

\bibitem[DO14a]{do1}
Christopher Davis and Tommy Occhipinti.
\newblock Which alternating and symmetric groups are unit groups?
\newblock {\em J. Algebra Appl.}, 13(3):1350114, 12, 2014.

\bibitem[DO14b]{do2}
Christopher Davis and Tommy Occhipinti.
\newblock Which finite simple groups are unit groups?
\newblock {\em J. Pure Appl. Algebra}, 218(4):743--744, 2014.

\bibitem[Fuc60]{fuchsprob}
L.~Fuchs.
\newblock {\em Abelian groups}.
\newblock International Series of Monographs on Pure and Applied Mathematics.
  Pergamon Press, New York-Oxford-London-Paris, 1960.

\bibitem[Gil63]{gilmer}
Robert~W. Gilmer, Jr.
\newblock Finite rings having a cyclic multiplicative group of units.
\newblock {\em Amer. J. Math.}, 85:447--452, 1963.

\bibitem[Lam01]{lam}
T.~Y. Lam.
\newblock {\em A first course in noncommutative rings}, volume 131 of {\em
  Graduate Texts in Mathematics}.
\newblock Springer-Verlag, New York, second edition, 2001.

\bibitem[Mac96]{fg2x2}
George Mackiw.
\newblock Finite groups of {$2\times 2$} integer matrices.
\newblock {\em Math. Mag.}, 69(5):356--361, 1996.

\bibitem[PS70]{PS}
K.~R. Pearson and J.~E. Schneider.
\newblock Rings with a cyclic group of units.
\newblock {\em J. Algebra}, 16:243--251, 1970.

\bibitem[Szy97]{bial}
Kazimierz Szymiczek.
\newblock {\em Bilinear algebra}, volume~7 of {\em Algebra, Logic and
  Applications}.
\newblock Gordon and Breach Science Publishers, Amsterdam, 1997.
\newblock An introduction to the algebraic theory of quadratic forms.

\end{thebibliography}

\end{document}